\documentclass[a4paper,12pt,openany,reqno]{amsart}
	\usepackage[top=2.5cm, bottom=2.5cm, left=2.5cm, right=2.5cm]{geometry}
\newcommand{\eu}{{e_+}}  \newcommand{\es}{{e_-}}

\newcommand{\BC}{\mathbb{BC}}
\newcommand{\R}{\mathbb{R}} 

\newcommand{\C}{\mathbb{C}}
\newcommand{\D}{\mathbb{D}}
 
\newcommand{\bz}{\overline{z}}

\usepackage[pdftex]{hyperref}
 \numberwithin{equation}{section}  \makeatletter\@addtoreset{equation}{section}
\usepackage{subfigure}% Support for small, `sub' figures and tables

\theoremstyle{plain}

\usepackage[english]{babel} \usepackage{latexsym,amsfonts,amsmath,mathrsfs}
 \newtheorem{thm}{Theorem}[section]
 \newtheorem{cor}[thm]{Corollary}
 \newtheorem{lem}[thm]{Lemma}
 \newtheorem{prop}[thm]{Proposition}
 \theoremstyle{definition}
 \newtheorem{defn}[thm]{Definition}
 \theoremstyle{remark}
 \newtheorem{rem}[thm]{Remark}
 \newtheorem*{ex}{Counterexample}
 \numberwithin{equation}{section}

\begin{document}

%-------------------------------------------------------------------------
% editorial commands: to be inserted by the editorial office
%
%\firstpage{1} \volume{228} \Copyrightyear{2004} \DOI{003-0001}
%\seriesextraline{This is the Concrete Title of this Book\br H.E. R and S.T.C. W, Eds.}
%\firstpage{1}

\title[bc-polyharmonicity and polyholomorphy]
{Bicomplex polyharmonicity and polyholomorphy}

%----------Author 1
\author[A. El Gourari]{Aiad El Gourari} 
\address{%
	Lab. P.D.E., Algebra and Spectral Geometry,\newline
	Department of mathematics, Faculty of sciences, P.O.Box 133.\newline
	Ibn Tofail University in Kenitra; Morocco}
\email{aiad.elgourari@uit.ac.ma}
%\thanks{This work was completed with the support of our \TeX-pert.}
\author[A. Ghanmi]{Allal Ghanmi} 
\address{Analysis, P.D.E $\&$ Spectral Geometry - Lab. M.I.A.-S.I., CeReMAR, \newline
		%Department of Mathematics, 
		P.O. Box 1014,  Faculty of Sciences, \newline
		Mohammed V University in Rabat, Morocco}
	\email{allal.ghanmi@um5.ac.ma}
\author[I. Rouchdi]{Ilham Rouchdi}
\address{Facult\'e des Sciences Juridiques, Economiques et Sociales,  \newline
	Universit\'e Sidi Mohamed Ben Abdellah de Fes, \newline
	Dhar El Mehraz, BP.42, 	Atlas 30000,   Fes, Maroc}
 	\email{irouchdi@yahoo.fr}

\quad



\date{February 19, 2021}

\begin{abstract}
		In this paper, we are concerned with the bicomplex analog of the well-known result asserting that real-valued harmonic functions, on simply connected domains, are the real parts of holomorphic functions. 
		We show that this assertion, word for word, fails for bc-harmonic functions and we provide a complete characterization of bc-harmonic functions that are the hyperbolic real parts of a specific kind of bc-holomorphic functions. Moreover, we extend the result to bicomplex polyharmonic functions, which implies the introduction of specific classes of bc-polyholomorphic functions. 
\end{abstract}

%%% ----------------------------------------------------------------------
\maketitle
%%% ----------------------------------------------------------------------
%\tableofcontents
\section{Introduction and statement of main results}

Harmonic function theory plays an important role in various fields of mathematics and physics, notably in studying minimal surfaces, digital processing, elasticity theory, and electrical engineering \cite{BerensteinGay1991,Choquet1945,Clunie1984,Kneser1926,Lewy1936,Ransford1995}. 
A well-known result in this theory asserts that, on simply connected domains, harmonic functions are the real parts of holomorphic functions (unique up to adding a constant), and that the converse also holds true. 
Such a result plays a crucial role in both complex analysis and potential theory \cite{BerensteinGay1991,Ransford1995}, and is always used to establish the connection between holomorphic and harmonic theories, and therefore to obtain interesting results regarding these fields by exploiting this connection \cite{Nicolesco1936}.

The extension of this property to other contexts, like those involving quaternion or bicomplex numbers, is not considered in the literature, even though the analogs of the holomorphic notion exist (slice regularity and bc-holomorphy) and their theories are already well developed. Such an extension will shed light on the corresponding theories and will open the door for further investigations and applications, for example, by establishing the bicomplex analogs of Liouville-Picard-Hadamard and Harnack-Montel theorems, among others, giving birth to bicomplex potential theory.

In the present paper, we examine the bicomplex analog of the aforesaid assertion.
We show that only the real-valued constant functions are realizable as the real part of bicomplex-holomorphic functions. Alternatively, this fails for both real and hyperbolic-valued (non constant) bc-harmonic functions, unless we consider, instead, a specific subclass of hyperbolic-valued bicomplex $\Delta_{bc}$-harmonic functions. More generally, we prove the following result. 

\begin{thm}\label{thmp000} 
	A hyperbolic-valued $\Delta_{bc}$-harmonic function in $\BC$  is the hyperbolic real part of a bc-holomorphic function if and only if it belongs to $\ker(\partial_{\widetilde{Z}}) \cap \ker(\partial_{Z^\dagger}) $.
\end{thm}  

The generalization of Theorem \ref{thmp000} to bc-polyholomorphic functions 
suggests the introduction of two specific classes $\mathcal{A}^{[1]}_{m,n}(\BC)$ and $\mathcal{A}^{[2]}_{m,n,k}(\BC)$ of bc-polyholomorphic functions (see Definitions \ref{defbcpolhol1kind} and \ref{defmnkbcpolyh}), for which we provide polynomial and operational characterizations.   
We next define the bicomplex polyharmonic functions, prove Theorem \ref{thmpolharm5} concerning their representation in terms of bc-harmonic functions, and discuss the bicomplex polyharmonicity of the real and hyperbolic real parts of the considered classes of bc-polyholomorphic functions as well as of their idempotent components (Proposition \ref{proppolholharm}). This will be very useful to connect the polyharmonic functions with harmonic functions.   
We also deal with the question concerning the determination and the uniqueness of such classes of bc-polyholomrphic functions for which the required assertion holds true (see Proposition \ref{propunic} and Theorem \ref{thmpolholharm5}).
To summarize, the main result provides a complete  and explicit description of those that are the hyperbolic real parts of bc-polyholomorphic functions.   

\begin{thm}\label{Mainthmp}
	Let $F: \BC \longrightarrow \D$ be a bc-polyharmonic function of (exact) order $m$. Then 
	\begin{enumerate}
		\item[(i)] If $F=\Re_{hyp} f$ for some $f \in \mathcal{A}^{[2]}_{s,n,k}(\BC)$, then $s=m$ and  
		$$F\in  \ker(\partial_{Z^\dagger}^{\max(n,k)}) \cap \ker(\partial_{\widetilde{Z}}^{\max(n,k)}).$$ 
		
		\item[(ii)] If $F\in  \ker(  \partial_{Z^\dagger}^{n}) \cap \ker(\partial_{\widetilde{Z}}^{k})$, then there exist a family of bc-polyharmonic functions $G_{\ell_1,\ell_2}$ of order $m$ decomposable as 
		$  G_{\ell_1,\ell_2}(Z,Z^*) = a_{\ell_1,\ell_2}(\alpha,\overline{\alpha} ) \eu + b_{\ell_1,\ell_2}(\beta,\overline{\beta}) \es $
		such that 
		$$ F(Z) = \sum_{\ell_1=0}^{n-1}\sum_{\ell_2=0}^{k-1} G_{\ell_1,\ell_2}(Z,Z^*) (Z^\dagger)^{\ell_1} \widetilde{Z}^{\ell_2}.$$
		Moreover, $F$ is of the form 
		$$ F(Z) = \sum_{\ell_1=0}^{n-1}\sum_{\ell_2=0}^{k-1} \Re_{hyp} f_{\ell_1,\ell_2} (Z,Z^*)   (Z^\dagger)^{\ell_1} \widetilde{Z}^{\ell_2}$$
		for some $ f_{\ell_1,\ell_2}\in \mathcal{A}^{[1]}_{r,s}(\BC)$ such that $m=\max(r,s)$.
	\end{enumerate}
\end{thm}

 In fact, this generalizes the result established for the first class  $\mathcal{A}^{[1]}_{m,n}(\BC)$. 
%It simply reads as %  (Theorems \ref{thmCharmn})  
\begin{thm}\label{thmCharmn}
The hyperbolic real part $\Re_{hyp}(f)$ of  $f\in \mathcal{A}^{[1]}_{m,n}(\BC)$is  $\Delta_{bc}$-polyharmonic function of order $\max(m,n)$. Conversely, a sufficient condition for a hyperbolic-valued $\Delta_{bc}$-polyharmonic function $F$ of order $\ell$ to be the hyperbolic real part of some $(m,n)$-bc-polyholomorphic function is to satisfy $\partial_{Z^\dagger} F= \partial_{\widetilde{Z}} F=0$ with $\ell=\max(m,n)$.
\end{thm}

For $m=1$, we recover Theorem \ref{thmp000} concerning the hyperbolic-valued bc-harmonic functions. 

The content of this paper is structured as follows. 
Besides the introductory section, needed notations, concepts, and elementary results on bicomplex numbers, holomorphy, and harmonicity are collected in Section 2. 
Our first main result for bc-harmonic functions is proved in this section. 
Section 3 is devoted to the  two classes of bc-polyholomorphic functions and their different characterizations. 
Section 4 deals with the bc-polyholomorphic functions.  
The main results for polyholomorphic functions are discussed, stated, and proved in Section 5.

\section{Preliminaries, notations and motivation}

\subsection{Backgrounds on bicomplex numbers.} 

Bicomplex numbers  generalize complex numbers more narrowly and offer a commutative $4$D real algebra (alternative to the division algebra of quaternions).
Their space $\BC$ can be thought of as specific combination of two copies of  $\C=\{x+iy ;\, x,y\in \R\}$, the set of complex numbers with $i$ as its imaginary unit. 
In fact, in the abstract algebra, $\BC$ is a set of structured pairs of complex numbers $(z_1,z_2)$  constructed by the Cayley-Dickson process that defines the bicomplex conjugate as $(z_1,-z_2)$, so that the corresponding product satisfies the property of quadratic form which means that $\BC$ is a composition algebra \cite{Dickson1914}.  
Subsequently, the linear algebra $\BC$,  endowed with 
the natural addition and the multiplication operations,  inherits similar properties as $\C$ except for division, where divisors of zero occur.
These numbers have been extensively studied by the Italian school of Segr\`e, Burgatti, Spampinato and Dragoni \cite{Segre1892,Burgatti1922,Dragoni1934,Spampinato1935a,Spampinato1936}.  For a complete study, we refer to  the works carried out by  Riley \cite{Riley1953} and Price \cite{Price1991}.  For extensive bibliography see
\cite{Riley1953,Hille1948,Takasu1943,Ward1940}.
 To provide the basis for a rigorous study of the modules of bc-holomorphic functions, many effort have been employed leading to general theory of functional analysis with bicomplex scalars \cite{AlpayElizarrarasshapiroStruppa2014}. 
Their properties and relationships with hypercomplex functions have been recently investigated in \cite{Ronn2001}. Some elementary functions are introduced and studied in \cite{RochonShapiro2004,AlpayElizarrarasshapiroStruppa2014,LunaShapiroStruppaVajiacBook2015}.
The associated  infinite and finite dimensional Hibertian structures have been considered  in \cite{LavoieMarchildonRochon2010,LavoieMarchildonRochon2011}. Related integral transforms, including bicomplex analogs of Fourier-Wigner, Segal-Bargmann and fractional Fourier transforms,
have been investigated in \cite{ElGourariGhanmiZine2020,GhanmiZine2019}.

The set  $\BC$ can be represented as  $ \BC=\{Z=z_1 + j z_2; z_1,z_2\in \C\}$, where  $j$ is another imaginary unit, $j^2=-1=i^2$, commuting with $i$ (i.e., $ij=ji$).
While the  matrix representation reads 
$$	\left(\begin{array}{cc}	z_1 & iz_2 \\ iz_2 & z_1
\end{array}\right)  $$
whose determinant $\det (Z)= z_1^{2}+z_2^{2}$.
Thus, given a bicomplex number $Z= z_1 + j z_2\in \BC$, 
we define its  bicomplex conjugate with respect to $j$ by
$$Z^\dagger=\mathcal{C}^{j}(Z) := z_1 - j z_2.$$
The conjugation with respect to $i$ is defined by
$$\widetilde{Z}=\mathcal{C}^{i}(Z) := \overline{z_1} + j \overline{z_2}$$
The last one $$ Z^*=\mathcal{C}^{j}(\mathcal{C}^{i}(Z)):= \overline{z_1} - j \overline{z_2}$$ 
is defined to be the conjugation with respect to both $i$ and $j$. 
The nullity of $\det (Z) = ZZ^\dag$ gives rise to specific classes of bicomplex numbers.
Indeed, $ZZ^{\dag}\ne 0$ characterizes those that are invertible, While $ZZ^{\dag} = 0$ characterizes the zero divisor set in $\BC$  given by the null cone of bicomplex numbers $$\mathcal{NC}=\{\lambda (1\pm ij ); \, \lambda\in \C\}.$$ The particular idempotent elements
$$ \eu  = \frac{1+ij}2 \quad \mbox{and} \quad \es  = \frac{1-ij}2$$
satisfy  
$ e_+^2=\eu $, $e_-^2=\es$, $\eu +\es  =1$, $\eu -\es  =ij$ and $\eu \es =0,$
so that any $Z= z_1 + j z_2\in \BC$ can be rewritten in a unique way as
\begin{align}\label{ib}
	Z= (z_1 - i z_2) \eu  +  (z_1 + i z_2) \es  = \alpha \eu  +  \beta \es
\end{align}
with $\alpha=z_1 - i z_2,\beta=z_1 + i z_2\in \C$.  The complex numbers $\alpha$ and $\beta$ are in fact the eigenvalue of $Z$ in the matrix representation. Moreover, the expression \eqref{ib} represents the corresponding diagonal matrix. 
The previous conjugates of a given $Z=\alpha \eu  + \beta \es $  read respectively $$Z^\dag=\beta \eu  + \alpha \es , \quad \widetilde{Z}=\overline{\beta}\eu  + \overline{\alpha}\es  \quad\mbox{and} \quad Z^*= \overline{\alpha}\eu  + \overline{\beta}\es .$$
Accordingly, the product of bicomplex numbers $Z=\alpha \eu  + \beta \es $  and $W=\alpha' \eu  + \beta' \es $ is given by $ ZW = \alpha\alpha' \eu  + \beta \beta' \es  $.
% so that $Z^n = \alpha^n \eu  +  \beta^n \es  $, and subsequently the exponential $e^Z$ can be defined by $e^Z = e^{\alpha} \eu  +  e^{\beta} \es .$
More details on algebraic properties can be found in \cite{Riley1953,Price1991,CatoniBoccalettiCannataCatoniNichelattiZampetti2008}.

%We should point out that the idempotent decomposition is crucial in simplifying the computation with bicomplex numbers and reduces them to complex numbers. However, some precautions are needed. 

To complete our review on the aforementioned bicomplex numbers, we should notice that such numbers can be realized also as the complexification of the so-called hyperbolic numbers defined as 
$$ \D := \{x + y k ; x,y \in \R\}  = \{x\eu + y\es; x,y \in \R\}; \, k:=ij.$$
Thus, according to the structure of bicomplex numbers, we define naturally a second kind of real part, the hyperbolic part. Notice for instance that, if  we let $\Re (\xi)$ denotes, as usual, the real part of a complex number $\xi\in \C$, then for $Z= z_1 + j z_2 
%= x_1 + iy_1 + jx_2+ k y_2 
= \alpha \eu + \beta \es$, we define 
\begin{align}\label{realc}
	\Re_c(Z) % =  x_1 
	= \Re(z_1) = \frac 12 (\Re(\alpha) + \Re(\beta)) =\frac{1}{4}\left( Z+Z^\dag+ \widetilde{Z}+Z^*  \right) 
\end{align} 
to be the classical real part of $Z$. 
Analogously, we define hyperbolic part.

\begin{defn}
	We call hyperbolic (real) part of the bicomplex number $Z= z_1 + j z_2 = \alpha \eu + \beta \es$ the quantity given by  
	\begin{align}\label{realhyp} \Re_{hyp}(Z)  = \Re(\alpha) \eu  + \Re(\beta)\es = \Re(z_1) + \Im(z_2) k=\frac{1}{2}\left( Z+Z^*  \right)  .
	\end{align}	
\end{defn}

The next assertion (as well as their variants) is immediate. 

\begin{lem} 
	The following assertions hold trues.
	\begin{enumerate}
		\item $Z$ is real ($Z=\Re_c(Z)$) if and only if $Z=Z^* = Z^\dag$. This is also equivalent to $\alpha=\beta\in \R$.
		\item  $Z$ is hyperbolic ($Z=\Re_{hyp}(Z)\in \D$) if and only if $Z=Z^*$ (or also $Z^\dag = \widetilde{Z}$).
	\end{enumerate}
\end{lem}

\subsection{Bicomplex holomorphic functions}
As in \cite{Price1991}, a bicomplex-valued function $f=f_1 +j f_2 $, on an open set $U\subset \BC$, is said to be bicomplex holomorphic (bc-holomorphic for short)  at a point $Z_0\in U$, and we write $f\in \mathcal{BH}ol(U)$, if it admits a bicomplex derivative at $Z_0$, i.e., if the limit
$$
\lim\limits_{\substack{ H \to 0 \\ H \notin \mathcal{NC}}} \frac{f(Z_0 + H) - f(Z_0)}{H}
$$
exists and is finite.
This is equivalent to say that the $\C$-valued functions $f_1$ and $f_2$ are holomorphic in the variables $z_1,z_2$ with $Z=z_1+jz_2$ and satisfy the Cauchy-Riemann system
$$
\left( \begin{array}{cc} \partial_{z_1}   & - \partial_{z_2}  \\ \partial_{z_2}   &  \partial_{z_1} \end{array}\right)
\left( \begin{array}{c} f_1  \\ f_2 \end{array}\right)= \left( \begin{array}{c} 0  \\ 0 \end{array}\right),
$$
where the shorthand $\partial_z$ is used to denote the differential operator 
\begin{equation}\label{crdo}
	\partial_{z} =\frac{\partial}{\partial z} =\frac{1}{2} \left( \frac{\partial}{\partial x} - i \frac{\partial}{\partial y}\right) ; \quad z= x+iy \in\C.
\end{equation}

The following characterization of the bc-holomorphicity is given in \cite{Rochon2008} and shows that holomorphic functions on the bicomplex space are once again solutions of a system of linear constant coefficients differential equations. Namely, a function $f\in \mathcal{C}^1(U)$ is bc-holomorphic on $U$ if and only if $f$ satisfies the following three systems of differential equations
$$
\frac{\partial f}{\partial Z^{*}} = \frac{\partial f}{\partial Z^{\dag}}  = \frac{\partial f}{\partial \widetilde{Z}}  = 0,
$$
where
\begin{align} 
	\frac{\partial}{\partial Z^{*}} &:= \frac{\partial }{\partial \overline {z_1}} + j  \frac{\partial }{\partial \overline {z_2}} = 
	\frac{\partial }{\partial \overline {\alpha}} \eu +   \frac{\partial }{\partial \overline {\beta}} \es  ,\label{Wirtingerstar}\\ 
	\frac{\partial}{\partial Z^{\dag}} &:=  \frac{\partial }{\partial z_1} + j  \frac{\partial }{\partial z_2}  = 
	\frac{\partial }{\partial \beta} \eu +   \frac{\partial }{\partial \alpha} \es  , \label{Wirtingerdagger}\\ 
	\frac{\partial f}{\partial \widetilde{Z}} &:= \frac{\partial }{\partial \overline {z_1}} - j  \frac{\partial }{\partial \overline {z_2}} = 
	\frac{\partial }{\partial \overline {\beta}} \eu +   \frac{\partial }{\partial \overline {\alpha}} \es  . \label{Wirtingerbar}
\end{align}
The above system is the foundation of the theory of bc-holomorphic functions.
Accordingly, we have the following characterization. 

\begin{thm}[{\cite[Theorem 15.5]{Price1991}}] \label{thmCharbcHolm} 
	A bicomplex-valued function $f$ is bc-holomorphic  if and only if it is of the form 
	\begin{align}\label{decomp:holom}
		f(Z)=f(\alpha \eu  + \beta \es ) = \phi^+(\alpha) \eu  + \phi^-(\beta) \es ,
	\end{align}
	for certain holomorphic $\C$-valued functions $\phi^\pm:\C \longrightarrow \C$.
\end{thm}

This provides us with another key tool that we use to extend many known results in theory of holomorphic functions to the bicomplex setting. 

\subsection{bc-harmonic functions}
Harmonic functions are usually defined as those belonging to the kernel of a Laplace operator. Thus, for the complex plane $\C\equiv \R^2$ the Laplace operator reads 
\begin{align}	\label{lapz}
	\Delta_z=\frac 14 \left( \frac{\partial^{2}}{\partial x^{2}}+\frac{\partial^{2}}{\partial y^{2}}  \right) = \frac{\partial^{2}}{\partial z \partial \overline{z}}; \, \, z= x+iy. 
\end{align}
To define harmonic functions in the bicomplex setting, we need to specify the bicomplex Laplace operator we will be working with, as there are many possible second order differential operators that may play this role \cite{AlpayElizarrarasshapiroStruppa2014}. The existence of different conjugations suggests the consideration of the following
\begin{align*}
& \Delta_{1} := \frac{\partial^2}{\partial Z \partial Z^{*}},  
&\Delta_{2}:= \frac{\partial^2}{\partial Z \partial Z^\dagger} ,  
&  \,\,\,\,\,\, \Delta_{3}:= \frac{\partial^2}{\partial Z \partial \widetilde{Z}} , \,\, 
\qquad  \Delta_{4}:= \frac{\partial^2}{\partial Z^{*} \partial Z^\dagger},\\ 
&\Delta_{5} := \frac{\partial^2}{\partial Z^{*} \partial \widetilde{Z}} ,
&\Delta_{6} := \frac{\partial^2}{\partial Z^\dagger \partial \widetilde{Z}}, \,\, \,\,\,\,\,\,
& \Delta_{7} := \frac{\partial^2}{\partial  Z \partial Z^{*} }   +  \frac{\partial^2}{\partial Z^\dagger \partial \widetilde{Z}} , 
\end{align*}
where 
%$\frac{\partial}{\partial Z }$ stands for 
\begin{align} 
	\frac{\partial}{\partial Z} &:= \frac{\partial }{\partial z_1} - j  \frac{\partial }{\partial z_2} = 
	\frac{\partial }{\partial \alpha} \eu +   \frac{\partial }{\partial \beta} \es  \label{Wirtinger}.
\end{align}
%It should be mentioned here that we have 
%\begin{align*} \Delta_{3} 
%	&=\frac{\partial^2}{\partial \alpha \partial \overline{\beta}} \eu + \frac{\partial^2}{\partial \beta \partial \overline{\alpha}} \es 
%	\\& = \left( \frac{\partial^2}{\partial z_1 \partial \overline{z_1}} - \frac{\partial^2}{\partial z_2 \partial \overline{z_2}}
%	\right)  Id_{\BC} +  2i \sigma_{\BC}  \Re  \left( \frac{\partial^2}{\partial z_1 \partial \overline{z_2}}\right) 
%\end{align*}
%and 
%\begin{align}
%\label{lap7} \Delta_{7} 
%= \Delta_{bc} + \Delta_{6} &= \Delta_{bc} + (\Delta_{bc})^\dagger= \Delta_{\alpha}+\Delta_{\beta}.
%\end{align} 
%While $\Delta_{2}$ is exactly the classical Laplace operator on $\C^2$ in $z_1$ and $z_2$, to wit 
%\begin{align*}
%%	\label{lap2} 
%\Delta_{2} &:= \frac{\partial^2}{\partial Z \partial Z^\dagger}  = \frac{\partial^2}{\partial \alpha \partial \beta}   = \frac{\partial^2}{\partial z_1 \partial \overline{z_1}} + \frac{\partial^2}{\partial z_2 \partial \overline{z_2}}.
%\end{align*} 

%In the sequel, when the index $s=1$ is we attributed to  $\Delta_s$, it means that we are dealing with $\Delta_{bc}$, i.e., $\Delta_{bc}=\Delta_{1}$. 

\begin{defn}
	A sufficiently differentiable bicomplex-valued function $f= f^+ \eu + f^-\es$ is said to be bicomplex harmonic with respect to the bicomplex Laplace operator $\Delta_{\bullet}$ (or $\Delta_{\bullet}$-harmonic for short), if it belongs to the kernel of such operator, i.e., $\Delta_{\bullet}f=0$.  
\end{defn}

	 Lemma \ref{lemreducytion} below is a reduction result. It answers the eventual question of what 
kind of bc-harmonicity one obtains when considering one of the bicomplex Laplace operators
$\Delta_{s}$, $s=1,2,\cdots,7$.  
To this end, we begin by defining  the operational bicomplex  conjugation  $T{^{*_{op}}}$,   $T{^{\dagger_{op}}}$ and  $T{^{\widetilde{}_{op}}}$ for a bicomplex operators $T$  satisfying $T(\alpha f +\beta g)=\alpha T(f) + \beta T(g)$ for $\alpha,\beta\in \BC$. 
%They can be seen as special extension of bicomplex conjugates for complex numbers.
%\begin{rem}\label{remcomconj}
%This is to not be confused with the ones $T{^{*_{s}}}$,   $T{^{\dagger_{s}}}$ and  $T{^{\widetilde{}_{s}}}$ defined by making standard composition of the considered operator and the corresponding bicomplex conjugation  $\mathcal{C}^{\bullet}$; 
%$\bullet= i,j,k$, which clearly satisfies the invariant property  $$ T{^{\bullet_{s}}}  := \mathcal{C}^{\bullet} \circ T  =   T \circ  \mathcal{C}^{\bullet}  .$$
%\end{rem}
Thus, for the bicomplex operators on bicomplex functional spaces taking the form $ T= A_1 +jA_2$ for certain  complex operators $A_1$ and $A_2$, we define 
 $$T{^{*_{op}}} = \left( A_1 +jA_2  \right){^{*_{op}}} := \overline{A_1} - j \overline{A_2}  ,$$   $$T{^{\dagger_{op}}} = \left(A_1 +jA_2 \right){^{\dagger_{op}}}  := A_1 - jA_2 $$ and  $$T{^{\widetilde{}_{op}}} = \left( A_1 +jA_2 \right) {^{\widetilde{}_{op}}} := \overline{A_1} + j \overline{A_2} .$$  
 The suggested operational bicomplex conjugates satisfy
 \begin{align} \label{conj4}
 	&\left( T{^{\bullet_{op}}}\right) ^{\bullet_{op}} = T ,
 \end{align}
 as well as the rotational rules ${^{{\bullet^1}_{op}}}{^{{\bullet^2}_{op}}}={^{{\bullet^3}_{op}}} $. That is
\begin{align} 
&	\left( {T^{*_{op}}}\right) ^{\dagger_{op}} = \left( {T^{\dagger_{op}}}\right) ^{*_{op}}=T{^{\widetilde{}_{op}}},\label{conj1} \\ 
&\left({T^{*_{op}}}\right)^{\widetilde{}_{op}} =\left( {T{^{\widetilde{}_{op}}}}\right) ^{*_{op}}= T{^{\dagger_{op}}}, \label{conj2}\\
&\left( {T^{\dagger_{op}}}\right) ^{\widetilde{}_{op}}=\left( {T{^{\widetilde{}_{op}}}}\right) ^{\dagger_{op}}= T^{*_{op}}. \label{conj3} 
\end{align}
Accordingly, 
%\begin{rem}
the elementary bicomplex first order differential operators in \eqref{Wirtingerstar}, \eqref{Wirtingerdagger} and \eqref{Wirtingerbar} are exactly the different operational bicomplex conjugate of the one in \eqref{Wirtinger}. Indeed, we have 
\begin{align} \label{conj4z}
	\frac{\partial}{\partial Z^{*}}  
	= \left( 	\frac{\partial}{\partial Z}\right) {^{*_{op}}}, \,\, 
	\frac{\partial}{\partial Z^{\dag}}
	= \left( 	\frac{\partial}{\partial Z}\right){^{\dagger_{op}}},\,\,  
	\frac{\partial f}{\partial \widetilde{Z}} 
	= 
	\left( 	\frac{\partial}{\partial Z}\right){^{\widetilde{}_{op}}} . %\label{Wirtingerbars}
\end{align}
%\end{rem}

\begin{rem}\label{remact}
The idempotent decomposition of $ T= A_1 +jA_2$ is given by $ T= A_+ \eu + A_-\es$, where $A_+=  A_1 -iA_2$ and $ A_-= A_1 + iA_2$. The action of $T$ on the bicomplex-valued function $f=f^+\eu + f^-\es$; $f^\pm=f_1 \mp if_2:\BC \longrightarrow \C$ is described as follows $$ T f= A_1f +jA_2f = A_+ f \eu + A_-f\es = A_+f^+ \eu + A_- f^-\es.$$
\end{rem}

\begin{lem} \label{lemreducytion}
	The bicomplex harmonicity with respect to $\Delta_{r}$ and $\Delta_{s}$ for positive integers $r$ and $r$ such that $r+s=7$ are equivalent. More precisely, if $f$ is a bicomplex-valued function, then  
	\begin{itemize}
		\item[(i)] $f$ is $\Delta_{6}$-harmonic if and only if $f^\dagger$ is $\Delta_{1}$-harmonic.
		\item[(ii)] $f$ is $\Delta_{5}$-harmonic if and only if $f^*$ is $\Delta_{2}$-harmonic.
		\item[(iii)] $f$ is $\Delta_{4}$-harmonic if and only if $f^*$ is $\Delta_{3}$-harmonic.
	\end{itemize}
\end{lem}

\begin{proof}
	Using the operational bicomplex conjugation rules \eqref{conj1}, \eqref{conj2}, \eqref{conj3} and \eqref{conj4} for bicomplex operators, it is clear that the considered bicomplex Laplace operators are connected to each other by
	\begin{align}
		&	\Delta_{1} = \Delta_{1}^{*} = \widetilde{\Delta_{6}} = \Delta_{6}^\dagger     \label{id1} \\
&	\Delta_{2}^{*}  = \widetilde{\Delta_{2}}   =\Delta_{5} = \Delta_{5}^\dagger, \label{id2}  \\    
&	\Delta_{3}^\dagger   =
	\Delta_{3}^{*}   =  \Delta_{4} = \widetilde{\Delta_{4}}, \label{id3}   
	.
\end{align}
	Moreover, using \eqref{id1} we get  
	$$\Delta_{6} f  = \widetilde{\Delta_{6}} ^\dagger f  = (\widetilde{\Delta_{6}} f^\dagger)^\dagger = ( \Delta_{1} f^\dagger)^\dagger  .$$
	Therefore, a bicomplex-valued function 
	 $f$ is $\Delta_{6} $-harmonic if and only if $f^\dagger$ is $\Delta_{1}$-harmonic (proving $(i)$).	 
	 In a similar way, by means of \eqref{id2} we get
	$$    
	\Delta_{5} f = \widetilde{\Delta_{5}} ^* f    =  (\widetilde{\Delta_{5}} f^*)^* = ( \Delta_{2} f^*)^* , $$
 so that $f$ is $\Delta_{5} $-harmonic if and only if $f^*$ is $\Delta_{2}$-harmonic.	This proves $(ii)$. Assertion $(iii)$ follows since from \eqref{id3} one obtains 
 $$    
 \Delta_{4} f = (\Delta_{4}^\dagger) ^* f    =  ( \Delta_{4}^\dagger f^*)^* = ( \Delta_{3} f^*)^*. $$ 
\end{proof}

\begin{rem}
The identities  \eqref{id1}, \eqref{id2}  and \eqref{id3} read equivalently 
$$\Delta_{2} = \Delta_{2}^\dagger=\Delta_{5}^{*}  = \widetilde{\Delta_{5}}   , \\    
\Delta_{3} = \widetilde{\Delta_{3}} = \Delta_{4}^\dagger   =
\Delta_{4}^{*}   , \\
\Delta_{6} = \Delta_{6}^{*}  = \widetilde{\Delta_{1}} = \Delta_{1}^\dagger   
.$$
\end{rem}

%\begin{proof}
%	This readily follows by observing that the considered operators are connected to each other.
%	More exactly, we have 
%	$$\Delta_{2}^{*}  = \widetilde{\Delta_{2}}   =\Delta_{5} ;\,\,    
%	\Delta_{3}^\dagger   =
%	\Delta_{3}^{*}   =  \Delta_{4} ; \,\,
%	\widetilde{\Delta_{6}} = \Delta_{6}^\dagger   =\Delta_{bc}
%	.$$
%	These identities readily follow from the definition of such bc-Laplacians and the operational bicomplex conjugation  of bicomplex operators defined above by direct use of the well-established rules in \ref{conj1}, \ref{conj2}, \ref{conj3} and \ref{conj4}.  
%\end{proof}
%
%\begin{rem}
%	We have 	$\Delta_{2}^\dagger   = \Delta_{2} ;\,\,  \widetilde{\Delta_{3}} =\Delta_{3}  $, $ 
%	\widetilde{\Delta_{4}}   =\Delta_{4}$, $ 
%	\Delta_{5}^\dagger   =\Delta_{5} $, $
%	\Delta_{6}^{*}   = \Delta_{6} $.  
%	\end{rem}
		
\begin{rem}
	Lemma \ref{lemreducytion} reveals that we need only to study the bicomplex harmonicity  with respect to the 
	bicomplex Laplace operators $\Delta_{1}$, $\Delta_{2}$ and $\Delta_{3}$. 
	We have excluded $\Delta_{7}$ for being reducible to the classical two-dimensional Laplacian in   $\alpha,\beta$. Indeed, we have 
	  \begin{align}
	 	\label{lap7} \Delta_{7} 
	 	= \Delta_{bc} + \Delta_{6} &= \Delta_{bc} + (\Delta_{bc})^\dagger= \Delta_{\alpha}+\Delta_{\beta}.
	 \end{align}
 Notice also that the operator $\Delta_{2}$ is purely complex  since  
 \begin{align*}
 	%	\label{lap2} 
 	\Delta_{2} &:= \frac{\partial^2}{\partial Z \partial Z^\dagger}  = \frac{\partial^2}{\partial \alpha \partial \beta}   = \frac{\partial^2}{\partial z_1^2 } + \frac{\partial^2}{\partial z_2^2},
 \end{align*}
 and can be seen as a specific Cartesian  Laplace operator on $\C^2$ that needs to be studied in its complex context. 
 Accordingly, we claim that in the bicomplex setting they are essentially two bicomplex Laplace operators to be studied,  $\Delta_{1}$ and $\Delta_{3}$. 
 \end{rem}

\begin{rem} 	
	%	   
	% 
	% It should be mentioned here that
	The explicit expression of $\Delta_{3}$ in the $z_1$ and $z_2$ variables is given by  
	 \begin{align*} \Delta_{3} 
	 %	&=\frac{\partial^2}{\partial \alpha \partial \overline{\beta}} \eu + \frac{\partial^2}{\partial \beta \partial \overline{\alpha}} \es 	\\
	 	& = \left( \frac{\partial^2}{\partial z_1 \partial \overline{z_1}} - \frac{\partial^2}{\partial z_2 \partial \overline{z_2}}
	 	\right)  Id_{\BC} +  2i \sigma_{\BC}  \Re  \left( \frac{\partial^2}{\partial z_1 \partial \overline{z_2}}\right) ,
	 \end{align*}
 where $Id_{\BC}$ is the identity operator and $ \sigma_{\BC}$ represents the multiplication operator by $k=ij$, $$\sigma_{\BC} (f^+\eu+f^-\es) = k (f^+\eu+f^-\es) = f^+\eu-f^-\es.$$ 
	\end{rem}

From now on, we privilege the Laplace operator $\Delta_{bc}:= \Delta_{1}$ which can be seen as a bicomplex analog of the complex Laplace operator in \eqref{lapz}. Accordingly, we preserve the notion of bc-harmonic functions to those taken with respect to $\Delta_{bc}$.  %Throughout this paper, we denote by $\Delta_{bc}$ the following  
%$$\Delta_{bc} := \frac{\partial^2}{\partial Z \partial Z^{*}} , $$
%where $\frac{\partial}{\partial Z }$ stands for 
%
%\begin{align} 
%	\frac{\partial}{\partial Z} &:= \frac{\partial }{\partial z_1} - j  \frac{\partial }{\partial z_2} = 
%	\frac{\partial }{\partial \alpha} \eu +   \frac{\partial }{\partial \beta} \es  \label{Wirtinger}.
%\end{align} 
Its explicit expression with respect to the idempotent decomposition simply reads  
$$\Delta_{bc} :=  \Delta_{\alpha}e^{+}+\Delta_{\beta}e^{-},$$
while the one in terms of the initial Cartesian coordinates $z_1$ and $z_2$ with $Z=z_1+jz_2$
%, makes appeal to the classical Laplace operator $\Delta_{\C^2}^{z_1,z_2}$ on the $2$D complex space $\C^2$. More precisely, we have 
is given by
\begin{align*}
	\Delta_{bc} 
	%&= \Delta_{\C^2}^{z_1,z_2} Id_{\BC} +  2 \sigma_{\BC}  \Im  \left( \frac{\partial^2}{\partial z_1 \partial \overline{z_2}}\right)\\
	& = \left(  \frac{\partial^2}{\partial z_1 \partial \overline{z_1}} +  \frac{\partial^2}{\partial z_2 \partial \overline{z_2}}  \right) Id_{\BC} +  2 \sigma_{\BC}  \Im  \left( \frac{\partial^2}{\partial z_1 \partial \overline{z_2}}\right)  .
\end{align*}

%%%%%%%%%%%%%%%%%%%%%%%%%%%%%%%%%%%%%%%%%%%%%%%%

%The following example will clarify the motivation of the present investigation. 
%
With the above data we can now discuss our motivating counterexample showing that the classical assertion concerning for the real-valued bc-harmonic functions fails.

\begin{ex}	
 Consider the function  $F_1(\alpha e^{+}+\beta e^{-})=(\alpha+\overline{\alpha})(\beta+\overline{\beta})$ on $\BC$, which is clearly a real-valued function and $\Delta_{bc}$-harmonic, $\Delta_{bc}F_1=0$. 
Thus, if we assume that $F_1=\Re_c(\psi) = \frac 12 \Re\left(\psi^+(\alpha) +\psi^-(\beta) \right)$ for certain bc-holomorphic function  $\psi(\alpha \eu +  \beta \es)=\psi^+(\alpha) \eu + \psi^-(\beta) \es$ with $\psi^\pm 
%:\C\longrightarrow \C 
\in \mathcal{H}ol(\C)$, and take the action of $\partial_{\overline{\alpha}}f$ and  $\partial_{\overline{\beta}} f$ (to be understood in the sense of \ref{remact}, since $\partial_{\overline{\alpha}}  = \partial_{\overline{\alpha}} \eu + \partial_{\overline{\alpha}} \es$), we get a contradiction. Indeed,
$$1= \partial_{\overline{\alpha}} \partial_{\overline{\beta}} F_1 = \frac 12 \partial_{\overline{\alpha}} \partial_{\overline{\beta}}  \Re\left(\psi^+(\alpha) +\psi^-(\beta) \right)= 0.$$
Although, the considered function can be realized as the real part of  $$G_1(\alpha e^{+}+\beta e^{-}) := 2\left( \overline{\alpha}(\beta+\overline{\beta}) \eu +\overline{\beta}(\alpha+\overline{\alpha})\es\right) = 2(\varphi\eu +\psi \es) ,$$ where $\varphi(\alpha,\beta) := \overline{\alpha}(\beta+\overline{\beta})$ and  $\psi(\alpha,\beta):=\overline{\beta}(\alpha+\overline{\alpha})$.
The constructed bicomplex-valued function $G_1$  is not bc-holomorphic (at least by direct computation). Moreover, we have  
$\Re_c(G_1)=2(\Re \varphi+ \Re \psi) = F_1$, 
which is bc-harmonic and belongs to $\ker \partial_{Z^*}^{2}$, $\ker\partial_{Z^\dagger}^{2}$ and $\ker\partial_{\overline{Z}}^{2}$. 
\end{ex}

This counterexample shows in particular that a  real-valued $\Delta_{bc}$-harmonic function does not need to be the real part of a bc-holomorphic function but perhaps a special kind of bicomplex polyanalytic function. 
To this aim, Theorem \ref{thmp000} discusses the existence of bc-holomorphic functions which are the real part of the given $\Delta_{bc}$-harmonic functions, and clarifies, partially, the previous counterexample. We present below a direct proof.

%\begin{thm}\label{thmp000} 
%	Let $F$ be a hyperbolic-valued $\Delta_{bc}$-harmonic function in $\BC$. Then, $F$ is the hyperbolic real part of a bc-holomorphic function if and only if $F$ belongs to $\ker(\partial_{\widetilde{Z}}) \cap \ker(\partial_{Z^\dagger}) $.
%\end{thm}  

\begin{proof}[Proof of Theorem \ref{thmp000}]
	Assume that $F=\Re_{hyp} f : \BC \longrightarrow \D$ for some given  bc-holomorphic, $\partial_{Z^*} f = \partial_{\widetilde{Z}} f = \partial_{Z^\dagger} f= 0$. By observing that this system is also equivalent to  
	%$\partial_{Z^*} f^\dagger = \partial_{\widetilde{Z}} f^\dagger = \partial_{Z} f^\dagger= 0$, 	$\partial_{Z^\dagger} \overline{f} = \partial_{Z} \overline{f} = \partial_{Z^*} \overline{f}= 0$ and 
	$\partial_{Z} f^* = \partial_{Z^\dagger} f^* = \partial_{\widetilde{Z}} f^*= 0$, we get  
	$ \partial_{Z^\dagger} F = \partial_{Z^\dagger} \Re_{hyp} f = \frac 12  \partial_{Z^\dagger}  \left(  f + f^* \right) = 0$ as well as $  \partial_{\widetilde{Z}} F = \frac 12 \partial_{\widetilde{Z}}  \left(  f + f^* \right) =0$. 	This proves $ F \in  \ker(\partial_{\widetilde{Z}}) \cap \ker(\partial_{Z^\dagger})$.
	Conversely, let $F: \BC \longrightarrow \D$ with
	$$ F(\alpha\eu + \beta\es) = F^+(\alpha\eu + \beta\es) \eu + F^-(\alpha\eu + \beta\es) \es$$ such that   $F \in \ker(\Delta_{bc})  \cap \ker(\partial_{\widetilde{Z}}) \cap \ker(\partial_{Z^\dagger})$. 
	From $\partial_{\widetilde{Z}} F= \partial_{Z^\dagger}F=0$, it is clear that the idempotent component $F^+:\BC \longrightarrow  \R$ is independent of $\beta$ and $\overline{\beta}$, while $F^-:\BC \longrightarrow  \R$ is independent of $\alpha$ and $\overline{\alpha}$. We write 
	$ F(\alpha \eu + \beta \es) = G^+(\alpha) \eu + G^-(\beta) \es,$
	where $G^+, G^-  :\C \longrightarrow  \R$ are real-valued functions in $\C$ defined by $G^+(\alpha) =  F^+ (\alpha \eu + \beta \es)$ and $G^-(\beta) = F^- (\alpha \eu + \beta \es)$, respectively.  
	Therefore, since the bc-harmonicity of $F$ is equivalent to the harmonicity in $\C$ of $ G^+$ and $G^-$,  with respect to $\Delta_\alpha$ and $\Delta_\beta$  respectively, there exist two holomorphic functions   $\varphi^+,\varphi^-:\C \longrightarrow \C$  such that  
	$ G^\pm   = \Re (\varphi^\pm)$.
	Therefore,   
	$$F(\alpha \eu +  \beta \es) = 
	\Re_{hyp}\left( \varphi^+(\alpha)  \eu +   \varphi^-(\beta)  \es \right) ,$$
	where $\varphi^+(\alpha)  \eu +   \varphi^-(\beta)  \es $ is clearly a bc-holomorphic thanks to Theorem \ref{thmCharbcHolm}.	
\end{proof}

Now, since for real-valued functions $F$ we have $F^+=F^-$, we conclude the following.

\begin{cor}\label{corcst}
	%	A real-valued $\Delta_{bc}$-harmonic function $F$ is not necessary the real part of a bc-holomorphic function. Moreover, 
	The only real-valued functions that are  bc-harmonic and realizable as the (standard) real part of bc-holomorphic functions are the constants.  
\end{cor}

\begin{rem}
This corollary justifies, some how,  the assertion of the counterexample discussed above.
\end{rem}

\section{bc-polyholomorphic functions (of finite order)}
We denote by $\mathcal{H}ol(\Omega)$ the space of holomorphic functions on an open set $\Omega$ in the complex plane $\C$, i.e., those $f:\Omega\longrightarrow \C$ satisfying the Cauchy-Riemann equation $\partial_{\bz} f=0$, where $\partial_{\bz}$ is as in \eqref{crdo}.
 As a specific generalization is the space of 
polyanalytic functions $f$ on $\Omega$ of order $n$ satisfying the generalized Cauchy-Riemann equation  $\overline{\partial}^{n} f =0$ (see \cite{Burgatti1922}), so that for $n=1$ we recover  $\mathcal{H}ol(\Omega)$.
This space is exactly the set of polynomial functions in $\overline{z}$ of degree less or equal to $n-1$ whose coefficients are holomorphic functions in $z$. 
   In the sequel, we deal with those that are of exact order $n$, 
$$  A_n^z[\overline{z}|\Omega] := \left\{ \sum_{k=0}^{n-1}a_{k}(z)\overline{z}^{k}; \, a_{k} \in \mathcal{H}ol(\Omega), \, a_{n-1} \not\equiv 0 , \, k=0,1,\cdots ,n-1\right\}.$$
We simply denote $A_n^z[\overline{z}]$ when $\Omega$ is the whole complex plane.
The sphere of intervention of these functions  includes many branches in mechanics and mathematical physics  \cite{Kolossov1908,Kolossov1909a,%Kolossov1909b,
	Muskhelishvili1977,Wendland1979,HaimiHedenmalm2013,AbreuFeichtinger2014}.
For a complete survey for their basic properties and their different applications one can refer to \cite{Balk1991} (see also \cite{AbreuFeichtinger2014}).

This notion of polyanalyticity can be extended in many natural ways to the bicomplex setting. Notice for instance that in analogy to $A_n^z[\overline{z}|\Omega]$, we set 
$$  A_n^Z[Z^*|U] := \left\{ \sum_{k=0}^{n-1} {Z^*}^{k}g_k  ;   g_k \in \mathcal{BH}ol(U),  g_{n-1}\not\equiv 0,    k=0,1,\cdots ,n-1\right\}$$
for a given open set $U$ in the bicomplex space $\BC$.

\begin{defn}\label{defbcnZpoly}
	The elements of $A_n^Z[Z^*|U]$ will be called $n$-$Z^*$-bc-polyholomorphic functions.
\end{defn}

Now, thanks to the idempotent decomposition we can suggest 
 $$ \mathcal{A}^{[1]}_{m,n}(U) :=  A_m^\alpha[\overline{\alpha}|U^+_\C] \eu + A_n^\beta[\overline{\beta}|U^-_\C] \es,$$
 i.e., the set of bicomplex-valued functions on $U$ that may be represented in the form
 $ f(z)= \varphi(\alpha) \eu + \psi(\beta)\es$ 
 for given $\varphi\in A_m^\alpha[\overline{\alpha}|U^+_\C]$ and $\psi\in A_n^\beta[\overline{\beta}|U^-_\C]$, 
where  $U^+_\C$ and $U^-_\C$ are the open sets in $\C$ defined by $U^+_\C=\{\alpha; \alpha\eu \in E^+ U\}$ and $U^-_\C=\{\beta; \beta\es \in E^- U\}$. Here $E^+$ and $E^-$ denote the continuous multiplication operators by $\eu$ and $\es$,  respectively.

\begin{defn}\label{defbcpolhol1kind}
	A bicomplex-valued function $f\in  \mathcal{A}^{[1]}_{m,n}(U) $ is called bicomplex polyholomorphic (or bc-polyholomorphic for short) on $U$  of order $(m,n)$ of first kind.
\end{defn}

	Notice for instance that from Definition \ref{defbcpolhol1kind}, it is easy to see that 
$\mathcal{A}^{[1]}_{m,n}(U) \cap \mathcal{A}^{[1]}_{j,k}(U) = \{0\}$ whenever $(m,n) \ne (j,k)$, so that 	$$\bigoplus_{m=0}^{r} \mathcal{A}^{[1]}_{m,\ell}(U) \quad \mbox{ and } \quad \bigoplus_{n=0}^{s} \mathcal{A}^{[1]}_{\ell,n}(U) $$ are direct sums, and so is 
$$  \bigoplus_{s_\ell=0}^{\ell-1} \left( \mathcal{A}^{[1]}_{s_\ell,\ell}(U)  \oplus   \mathcal{A}^{[1]}_{\ell,s_\ell}(U) \right) \oplus  \mathcal{A}^{[1]}_{\ell,\ell}(U).$$
The later one can be rewritten as 
$$  \left( \bigoplus_{m=0}^{\ell-1} \mathcal{A}^{[1]}_{m,\ell}(U)\right)  \oplus \left( \bigoplus_{n=0}^{\ell-1} \mathcal{A}^{[1]}_{\ell,n}(U) \right) \oplus  \mathcal{A}^{[1]}_{\ell,\ell}(U) = \bigoplus_{m,n=0\atop \max(m,n)=\ell}^\ell \mathcal{A}^{[1]}_{m,n}(U).$$
Moreover, we can prove the following.

\begin{prop}\label{propExpSer}
	For fixed positive integer $\ell$, the space $A_\ell^Z[Z^*|U]$ is the direct sum of $\mathcal{A}^{[1]}_{\ell,s_\ell}(U)$ and $\mathcal{A}^{[1]}_{\ell,s_\ell}(U)$ for
	 $s_\ell = 0,1, \cdots ,\ell$, 
	\begin{align}		\label{char1} 
		A_\ell^Z[Z^*|U] = \bigoplus_{m,n=0\atop \max(m,n)=\ell}^\ell \mathcal{A}^{[1]}_{m,n}(U).
	\end{align}
\end{prop}

%\begin{proof}
%	Notice first that from Definition \ref{defbcpolhol1kind}, it is easy to see that each $\mathcal{A}^{[1]}_{m,n}(U)$ is contained in $A_\ell^Z[Z^*|U]$ and that their sum involved in the right hand side of \eqref{char1} is direct.  Now, for given open domain $U\subset \BC$, let  $U^+_\C$ and $U^-_\C$ be the regions  in $\C$ defined by $U^+_\C=\{\alpha; \alpha\eu \in E^+ U\}$ and $U^-_\C=\{\beta; \beta\es \in E^- U\}$, where $E^+$ and $E^-$ denote the continuous multiplication operators by $\eu$ and $\es$,  respectively. Then, for every $$f(z) =  \sum_{k=0}^{\ell-1} {Z^*}^k g_k(Z)\in A_\ell^Z[Z^*|U],$$ with $g_k$ being bc-holomorphic functions in $U$, there exists nonnegative integers $m,n$ and holomorphic functions $a_{k}\in \mathcal{H}ol(U^+_\C) $; $k=0,1,\cdots, m-1 $ and $b_{k}\mathcal{H}ol(U^-_\C) $; $k=0,1,\cdots,n-1 $, such  that $\ell= \max(m,n)$ and $g_k(Z) = a_{k}(\alpha) \eu + b_{k}(\beta) \es$ for every $m=0,1,\cdots , \ell$. Therefore, 
%	\begin{align*} f(z) &=  \sum_{k=0}^{m-1}a_{k}(\alpha)\overline{\alpha}^{k} \eu + 
%		\sum_{k=0}^{n-1}b_{k}(\beta)\overline{\beta}^{k}\es  
%	\end{align*} 
%	which clearly belongs to $\mathcal{A}^{[1]}_{m,n}(U)$.  
%\end{proof}

\begin{proof}
From Definition \ref{defbcpolhol1kind}
	each $\mathcal{A}^{[1]}_{m,n}(U)$ is clearly contained in $A_\ell^Z[Z^*|U]$. In fact, for every $f \in \mathcal{A}^{[1]}_{m,n}(U)$ there exist some holomorphic functions $\varphi_k \in \mathcal{H}ol(U^+_\C)$ and $\psi_k \in \mathcal{H}ol(U^-_\C) $ such that $\varphi_{m-1} \not\equiv 0$, $\psi_{n-1} \not\equiv0$, and
\begin{align*}
	f(Z)& = \left( \sum_{k=0}^{m-1} \overline{\alpha}^k \varphi_k(\alpha)\right)  \eu +
	\left( \sum_{k=0}^{n-1} \overline{\beta}^k \psi_k(\beta) \right) \es \\
	 &= \sum_{k=0}^{\max(m,n)-1} \left( \overline{\alpha}^k a_k^m(\alpha)  \eu +
\overline{\beta}^k b_k^n(\beta)  \es\right) \\& =
\sum_{k=0}^{\max(m,n)-1} {Z^*}^{k}g_k(Z) ,
\end{align*}
	where we have set  
$$  \widetilde{\varphi}_k^m (\alpha)= \left\{ \begin{array}{ll}
	\displaystyle  \varphi_k (\alpha) , & k=0,1, \cdots, m-1 \\
	0  , & k=m, m+1, \cdots 
\end{array} \right. $$ and 
$$  \widetilde{\psi}_k^n (\beta)= \left\{ \begin{array}{ll}
	\displaystyle  \psi_k (\beta) , & k=0,1, \cdots, n-1 \\
	0  , & k=n, n+1, \cdots .
\end{array} \right. $$
 The involved functions $ g_k$ are given by  $ g_k(Z):=  \widetilde{\varphi}_k^m(\alpha)  \eu +   \widetilde{\psi}_k^n(\beta)  \es$ and belongs to $  \mathcal{BH}ol(U)$ thanks to Theorem \ref{thmCharbcHolm}. Moreover,
 $$ g_{\max(m,n)-1} (\alpha)= \left\{ \begin{array}{lll}
 	\displaystyle  \varphi_{m-1} (\alpha) \eu , & m > n , \\
 		\displaystyle  \varphi_{m-1} (\alpha) \eu + \displaystyle  \psi_{m-1} (\beta) \es, & m = n , \\
 			\displaystyle  \psi_{n-1} (\beta) \es , & m < n ,
 \end{array} \right. \not\equiv 0. $$ 
This proves that $f \in  A_\ell^Z[Z^*|U]$ for fixed $\ell$ with $\ell = \max(m,n)$, and therefore
$$ \bigoplus_{m,n=0\atop \max(m,n)=\ell}^\ell \mathcal{A}^{[1]}_{m,n}(U) \subset  	A_\ell^Z[Z^*|U] .$$

The converse is implicitly contained in the above lines. Indeed, if
$$f(Z)  =
\sum_{k=0}^{\ell-1} {Z^*}^{k}g_k(Z) \in   A_\ell^Z[Z^*|U]$$
for certain  $ g_k(Z):= a_k(\alpha)  \eu +  b_k(\beta)  \es $  bc-holomorphic functions in $U$, with $a_{k}\in \mathcal{H}ol(U^+_\C) $  and $b_{k}\in \mathcal{H}ol(U^-_\C) $, then from $g_{\ell-1} \not\equiv 0$ we have $a_{\ell-1} \not\equiv  0$ or $b_{\ell-1} \not\equiv  0$. Thus, let $m_\ell = \min\{k = 1, \cdots, \ell; \,  a_{k-1} \not\equiv 0\}$ and  $n_\ell = \min\{k = 1, \cdots, \ell; \,  b_{k-1} \not\equiv 0\}$. Then, $\max(m_\ell,n_\ell) =\ell$ and 
%
% is infers the existence of a pair of nonnegative integers 
%$$(m,n) \in L_\ell := \{ (\ell,k) ; \, k =0, 1, \cdots ,\ell\} \cup \{ (k,\ell) ; \, k =0, 1, \cdots ,\ell\}$$
%such that 
	\begin{align*} f(z) &=  \left( \sum_{k=0}^{m_\ell-1} \overline{\alpha}^{k} a_{k}(\alpha) \right) \eu + 
	\left( \sum_{k=0}^{n_\ell-1}\overline{\beta}^{k}b_{k}(\beta)\right) \es , 
\end{align*} 
which clearly belongs to $\mathcal{A}^{[1]}_{m_\ell,n_\ell}(U)$. 
%To conclude it suffices to observe that  $ L_\ell = \{(m,n); m,n=0,1,\cdots,\ell; \, \max(m,n)=\ell\} $.
This completes the proof. 
\end{proof}

\begin{rem} The case of $\ell=1$ leads to 
	$$ 	A_1^Z[Z^*|U] =   \mathcal{A}^{[1]}_{0,1}(U)  \oplus  \mathcal{A}^{[1]}_{1,0}(U)   \oplus  \mathcal{A}^{[1]}_{1,1}(U) .$$
	\end{rem}
\begin{rem}
	Let $f$ be a $(m,n)$-bc-holomorphic function in a certain  region $U\subset \BC$. Then, it can be expanded as 
	\begin{align*} f(z) &=  \sum_{k=0}^{\max(m,n)-1} {Z^*}^k g_k(Z)
	\end{align*} 
	for given  bc-holomorphic functions  $g_k$ in $U$.
\end{rem}

The  interpretation in terms of the first order differential operators  \eqref{Wirtingerstar}, \eqref{Wirtingerdagger} and \eqref{Wirtingerbar} is provided by the following assertion.

\begin{prop}\label{propmax1kind}
	We have
	\begin{align}\label{char2} 
		A_\ell^Z[Z^*|U] = \ker\left( \frac{\partial^{\ell }}{\partial {Z^*}^{\ell}}\right) \cap \ker\left(  \frac{\partial }{\partial  Z^\dagger}\right)  \cap \ker \left( \frac{\partial }{\partial \widetilde{Z}} \right).
	\end{align}
\end{prop}

\begin{proof}
	Thanks to Proposition \ref{propExpSer}, the assertion of Proposition \ref{propmax1kind} becomes equivalent to 
	\begin{align*} 
		\bigoplus_{m,n=0\atop \max(m,n)=\ell}^{\ell} \mathcal{A}^{[1]}_{m,n}(U) = \ker\left( \frac{\partial^{\ell }}{\partial {Z^*}^{\ell}}\right) \cap \ker\left(  \frac{\partial }{\partial  Z^\dagger}\right)  \cap \ker \left( \frac{\partial }{\partial \widetilde{Z}} \right).
	\end{align*}
	For $f=   f^+\eu+f^-\es\in \mathcal{A}^{[1]}_{m,n}(U)$, the component functions $f^\pm: U^\pm_\C \longrightarrow \C$ are polyanalytic of order $m$ and $n$, respectively. Therefore, they satisfy
	$$	\left\{ 
	\begin{array}{lll}
		\partial^{m}_{\overline{\alpha}}f^{+} & = 0   =  	\partial^{n}_{\overline{\beta}}f^{-} , \\
		\partial_{\overline{\beta}}f^{+} & = 0 =  	\partial_{\overline{\alpha}}f^{-},\\
		\partial_{\beta}f^{+} & = 0 =  \partial_{\beta}f^{-} .
	\end{array}
	\right.
	$$
	This can be rewritten equivalently as  
	\begin{align}\label{prop1kind} \frac{\partial^{\max(m,n)}f}{\partial {Z^*}^{\max(m,n)}}= \frac{\partial f}{\partial  Z^\dagger} = \frac{\partial f}{\partial \widetilde{Z}} =0 
	\end{align}
	and shows that 
	$$\mathcal{A}^{[1]}_{m,n}(U) \subset \ker\left( \frac{\partial^{\max(m,n)}}{\partial {Z^*}^{\max(m,n)}}\right) \cap \ker\left(  \frac{\partial }{\partial  Z^\dagger}\right)  \cap \ker \left( \frac{\partial }{\partial \widetilde{Z}} \right) .$$ 
	Now, if $g$ satisfies %\eqref{prop1kindinv}, 
	\begin{align}\label{prop1kindinv}
		\frac{\partial^{\ell}g}{\partial {Z^*}^{\ell}}= \frac{\partial g}{\partial  Z^\dagger} = \frac{\partial g}{\partial \widetilde{Z}} =0,
	\end{align}
	we obtain $	\partial^{\ell}_{\overline{\alpha}}f^{+}  =  	\partial^{\ell}_{\overline{\beta}}f^{-} =0$. Moreover, $f^{+}$ (resp. $f^{-}$) is clearly independent of $\beta$ and $\overline{\beta}$ (resp. $\alpha$ and $\overline{\alpha}$). Then, 
	%from which we need only to observe that 
	there exists a pair of nonnegative integers $(m,n)$ such that $\ell=\max(m,n)$, and $\partial^{m}_{\overline{\alpha}}f^{+} =0  =  	\partial^{n}_{\overline{\beta}}f^{-} $. Therefore, $g$ is a $(m,n)$-bc-polyholomorphic in $U$. 
\end{proof}

A more general definition can be proposed from  \eqref{Wirtingerstar}, \eqref{Wirtingerdagger} and \eqref{Wirtingerbar}. 
%The following one defines a second kind of bc-polyholomorphic functions.

\begin{defn}[\cite{ElGourariGhanmiZine2020}]\label{defmnkbcpolyh}
	A bicomplex-valued function $f: U\rightarrow \BC$   is said to be $(m,n,k)$-bc-polyholomorphic if
	$$\partial ^{m}_{Z^{\ast}}f=\partial ^{n}_{\widetilde{Z}}f    =\partial ^{k}_{Z^{\dagger} } f=0 . $$
	We denote by $\mathcal{A}^{[2]}_{m,n,k}(U)$ the corresponding space.
\end{defn}

The class of $(m,n)$-bc-polyholomorphic of first kind appears then as a subclass of the second one in Definition \ref{defmnkbcpolyh}, so that Proposition \ref{propmax1kind} can be reworded as 
$$ 
\mathcal{A}^{[2]}_{\ell,1,1}(U) 
= \bigoplus_{j,k=0 \atop
	\max(j,k)=\ell}^{\ell} \mathcal{A}^{[1]}_{j,k}(U) .$$

The following establishes a characterization of $(m,n,k)$-bc-polyholomorphic functions. 

\begin{prop}\label{propChar}
	A bicomplex-valued function $f$ in $\BC$ is $(m,n,k)$-bc-polyholomorphic if and only if it can be expanded as 
	\begin{align}\label{expansion}
		f(Z)=\sum\limits_{\ell_{1}=0}^{m-1}
		\sum\limits_{ \ell_{2}=0}^{n-1}
		\sum\limits_{\ell_{3}=0}^{k-1} {Z^{\ast}}^{\ell_{1}}\widetilde{Z}^{\ell_{2}}{Z^{\dagger} }^{\ell_{3}}H_{\ell_{1},\ell_{2},\ell_{3}}(Z) 
	\end{align}
	for given bc-holomorphic functions $H_{\ell_{1},\ell_{2},\ell_{3}}$.	
\end{prop}

\begin{proof}
	Write   $f(\alpha e^{+}+\beta e^{-})=f^{+}(\alpha e^{+}+\beta e^{-})e^{+}+f^{-}(\alpha e^{+}+\beta e^{-})e^{-}$  for some given $ f^{\pm}:\BC\rightarrow \C.$ Then, by means of the idempotent decomposition, the system $\partial ^{m}_{Z^{\ast}}f=\partial ^{n}_{\widetilde{Z}}f=\partial ^{k}_{Z^{\dagger} }f=0$ 
	reads equivalently as
	$ \partial^{m}_{\overline{\alpha}}f^{+}=
	\partial^{n}_{\overline{\beta}}f^{+} = 
	\partial^{k}_{\beta}f^{+} =0$ and 
	$\partial^{m}_{\overline{\beta}}f^{-} =
	\partial^{n}_{\overline{\alpha}}f^{-} =
	\partial^{k}_{\alpha}f^{+} =0$.
	Subsequently,	$f^{+}$ and $f^{-}$ can be expanded as
	$$ f^{+}(\alpha e^{+}+\beta e^{-})=\sum\limits_{\ell_{1}=0}^{m-1}
	\sum\limits_{ \ell_{2}=0}^{n-1}
	\sum\limits_{\ell_{3}=0}^{k-1}
	\overline{\alpha}^{\ell_{1}}\overline{\beta}^{\ell_{2}}\beta^{\ell_{3}}h_{\ell_{1}\ell_{2}\ell_{3}}(\alpha)$$ and
	$$ f^{-}(\alpha e^{+}+\beta e^{-})=\sum\limits_{\ell_{1}=0}^{m-1}
	\sum\limits_{ \ell_{2}=0}^{n-1}
	\sum\limits_{\ell_{3}=0}^{k-1}\overline{\beta}^{\ell_{1}}\overline{\alpha}^{\ell_{2}}\alpha^{\ell_{3}}g_{\ell_{1}\ell_{2}\ell_{3}}(\beta)$$
	for the given holomorphic functions $h_{\ell_{1},\ell_{2},\ell_{3}}, g_{h_{1},h_{2},h_{3}} :\C\longrightarrow \C$.
	This proves \eqref{expansion} by setting $H_{\ell_{1},\ell_{2},\ell_{3}}(Z) = h_{\ell_{1},\ell_{2},\ell_{3}}\eu +  g_{h_{1},h_{2},h_{3}}\es$, which is clearly bc-holomorphic thanks to Theorem \ref{thmCharbcHolm}. 
\end{proof}

%\begin{rem}	
An alternative proof of Proposition \ref{propChar} using its complex version is the following. % given in the Appendix. 
%\end{rem}
%\section{Appendix: An alternative proof of Proposition \ref{propChar} }
Starting from given
$$  f(\alpha e^{+}+\beta e^{-})=f^{+}(\alpha e^{+}+\beta e^{-})e^{+}+f^{-}(\alpha e^{+}+\beta e^{-})e^{-}$$
in $\mathcal{A}^{[2]}_{m,n,k}(T)$, we can claim that the partial functions
\begin{eqnarray}
	F^{+}_{\beta_{0}}: \C \rightarrow \C ; \, 
	\alpha \longmapsto  F^{+}_{\beta_{0}}(\alpha)=f^{+}(\alpha e^{+}+\beta_{0}e^{-})
\end{eqnarray}
and
\begin{eqnarray}
	F^{-}_{\alpha_{0}}: \C \rightarrow \C ; \, 
	\beta \longmapsto  F^{-}_{\beta_{0}}(\beta)=f^{-}(\alpha_{0} e^{+}+\beta_{0}e^{-})
\end{eqnarray}	
are polyanalytic with order $m$ for every fixed complex numbers $\beta_{0}$ and $\alpha_{0}$, respectively. In fact, we have
$ \partial^{m}_{\overline{\alpha}}F^{+}_{\beta_{0}}=\partial^{m}_{\overline{\alpha}}f^{+}(\alpha e^{+}+\beta_{0}e^{-})=0$, $ \partial^{m-1}_{\overline{\alpha}}F^{+}_{\beta_{0}} \ne 0$, 
$\partial^{m}_{\overline{\beta}}F^{-}_{\alpha_{0}}=\partial^{m}_{\overline{\beta}}f^{-}(\alpha_{0} e^{+}+\beta e^{-})= 0$ and $ \partial^{n-1}_{\overline{\beta}}F^{+}_{\alpha_{0}} \ne 0$, since $F^{+}_{\beta_{0}} \in A_m^\alpha[\overline{\alpha}]$ and $F^{-}_{\alpha_{0}}\in A_m^\beta[\overline{\beta}]$.
Therefore, from \cite{Balk1991}, we know that the real parts $\Re(F^{+}_{\beta_{0}})$ and $\Re(F^{-}_{\alpha_{0}})$ are $m$-polyharmonic with respect to the Laplacian $\Delta_{\alpha}$ and $\Delta_{\beta}$, respectively,
$$\Delta_{\alpha}^{m}\left(\Re (F^{+}_{\beta_{0}})\right)=0=\Delta_{\beta}^{m}\left(\Re (F^{-}_{\alpha_{0}})\right)$$ for every fixed $\alpha_{0}$, $\beta_{0}$.
Subsequently,
$$\left(\Delta_{\alpha}\Delta_{\beta}\right)^{m}(\Re (f)) 
=
\Re\left(\Delta_{\beta}^{m}\left(\Delta_{\alpha}^{m}F_{\beta}^{+}\right)+
\Delta_{\alpha}^{m}\left(\Delta_{\beta}^{m}F_{\alpha}^{-}\right)\right)=0.$$
This shows that $\Re_{c}(f):T\rightarrow \R$ is $\Delta_{\alpha}\Delta_{\beta}$- polyharmonic of order $m$. In a similar way, we obtain
$$\Delta_{bc}^{m}\left(\Re_{hyp} (f)\right)=\Delta_{\alpha}^{m}\left(\Re (F_{\beta}^{+})\right)e^{+}+\Delta_{\beta}^{m}\left(\Re (F_{\alpha}^{+})\right)e^{-}=0.$$
This complete the proof of Proposition \ref{propChar}.

\section{Bicomplex polyharmonic functions}

A complex-valued function $u(x,y)$ is said to be polyharmonic of order $n$ in some region $\Omega$ of the plane $\R^{2}$, if it belongs to $\mathcal{C}^{2n}(\Omega)$ and satisfies the Laplace equation of order (the smallest nonnegative) $n$, $ \Delta^{n}u(x,y)=0$, with $\Delta$ is the Laplace operator in \eqref{lapz}. Its general solution plays an important rule in the torsion-free axisymmetric deformation problems in elasticity  theory.
The close connection to polyanalytic remains valid as expected in the assertion below {\cite[p.19]{Balk1991}}.

\begin{thm} \label{polyHH}
	A complex-valued function $f=\phi + i\psi$ on $\C$; $\phi,\psi: \C \longrightarrow \R$, belongs to $A_n^z[\overline{z}]$ if and only if its real part $\Re(f):=\phi$ is a polyharmonic function of the same order. 
\end{thm}

An interesting result in the theory of polyharmonic functions is  due to Almansi \cite{Almansi1899}. Below we state it for the complex plane $\C=\R^2$. However, the result remains valid in for star-shaped %star-like 
domain with respect to the origin. high dimension.
% with $|x|$ is the Euclidean length of $x$.

\begin{thm}\label{thmAlmansi}
	Every polyharmonic complex-valued function $F$ of order $m$  admits a unique decomposition of the form 
	$$ f(z) = \sum_{k=0}^{m-1} |z|^{2k} h_k(z) ,$$
	where $h_k$ are harmonic functions. 
\end{thm}

In the bicomplex context, we suggest the following.

\begin{defn}
A bicomplex-valued function $F$ is said to be bicomplex polyharmonic   (shortly, bc-polyharmonic) of (exact) order $n$ if $n$ the smallest positive integer such that 
$\Delta_{bc}^n f =0$. The case of $n=1$ corresponds to the conventional bc-harmonic functions.
\end{defn}
Subsequently,  $F= F^+ \eu + F^-\es$ is $n$-bc-polyharmonic with respect to $\Delta_{bc}$ if and only if the idempotent components $F^+$ and $F^-$ are polyharmonic of order $n_1$ and $n_2$ with respect to $\Delta_{\alpha}$ and $\Delta_{\beta}$, respectively, with $n=\max(n_1,n_2)$.
The next theorem provides a representation
of the so-called bicomplex polyharmonic functions in terms of bc-harmonic functions. This is in fact a natural extension of  Almansi's result (Theorem \ref{thmAlmansi}) to the bc-polyharmonic functions
% stating that 

\begin{prop}\label{thmpolharm5}
	Let $F :\BC \longrightarrow \D$  be a $\Delta_{bc}$-polyharmonic function of order $m$. Then, there exists bc-harmonic functions $ H_k$, $k=0,1,\cdots, m-1$, such that 
	\begin{align}\label{exppolyharm}
		F(Z)  = 
		\sum_{k=0}^{m-1}  |Z|_{bc}^{2k} H_k(Z) .
	\end{align}
\end{prop}

\begin{proof}
	Let   $F$ 
	%$ F  (\alpha \eu + \beta \es)= F^+ (\alpha \eu + \beta \es) \eu +  F^- (\alpha \eu + \beta \es)\es $, 
	 be a hyperbolic-valued  polyharmonic function of order $m$ with respect to $\Delta_{bc}:= \Delta_\alpha\eu + \Delta_\beta \es$. Since the partial functions  
	$ F^+_{\beta_0}(\alpha) := F^+ (\alpha \eu + \beta_0 \es)$ and  $ F^-_{\alpha_0}(\beta) := F^- (\alpha_0 \eu + \beta \es)$ are polyharmonic  on $\C$ of order $r$ and $s$ with respect to the Laplacians $\Delta_\alpha$ and $\Delta_\beta$, for every $\beta_0$ and $ \alpha_0$ respectively, with $\max(r,s)= m$, one can apply Theorem \ref{thmAlmansi}. Thus, there exist certain harmonic functions
	$h^{\beta_0}_k$ and $g^{\alpha_0}_k$ with respect to $\Delta_\alpha$ and $\Delta_\beta$, respectively, such that 
	$$ \displaystyle F^+_{\beta_0}(\alpha) =  \sum_{k=0}^{m-1}  |\alpha|^{2k} h^{\beta_0}_k(\alpha) \quad \mbox{ and } \quad \displaystyle F^-_{\alpha_0}(\beta)=\sum_{k=0}^{m-1}  |\beta|^{2k}  g^{\alpha_0}_k(\beta).$$
	Thus, we can rewrite $F(Z)$  as in \eqref{exppolyharm}, 
	where 
	$ H_k(Z):= h^{\beta}_k(\alpha)\eu + g^{\alpha}_k(\beta) \es$ for $Z=\alpha\eu+\beta\es$, with  
	$$\Delta_{bc}  H_k =  \Delta_{\alpha} h^{\beta}_k(\alpha) \eu + \Delta_{\beta} g^{\alpha}_k(\beta) \es = 0.$$
	This completes the proof.
\end{proof}

%{\bf	Therefore, by Theorem \ref{polyHH}, there exist two families of polyanalytic functions $\varphi_\beta:\C \longrightarrow \C \in A_r^\alpha[\overline{\alpha}]$ and $\phi_\alpha:\C \longrightarrow \C \in A_s^\beta[\overline{\beta}]$
%	such that 
%	$  F^+_\beta = \Re (\varphi_\beta )
%	$ and  
%	$F^-_\alpha=\Re (\phi_\alpha).$
%Now, since 
%$$ \displaystyle \varphi_\beta(\alpha) =  \sum_{k=0}^{m-1}  \overline{\alpha}^k h^{\beta}_k(\alpha) \quad \mbox{ and } \quad \displaystyle \phi_\alpha(\beta)=\sum_{k=0}^{m-1}  \overline{\beta}^k g^{\alpha}_k(\beta)$$
%for certain holomorphic functions.
%This completes the proof.
%}

\section{Hyperbolic real part of bc-polyharmonic functions}
In this section, we discuss the problem of characterizing the hyperbolic-valued $\Delta_{bc}$-bc-polyharmonic functions that are the hyperbolic real parts of the bicomplex-valued polyholomorphic functions. To this purpose, we begin by giving the proof of the bicomplex analog related the $(m,n)$-bc-holomorphic functions (stated as Theorem \ref{thmCharmn} in the introductory section).  

%\begin{thm}\label{thmCharmn}
%	If $f\in \mathcal{A}^{[1]}_{m,n}(\BC)$, then $\Re_{hyp}(f)$ is   $\Delta_{bc}$-polyharmonic function of order $\max(m,n)$. Moreover, a sufficient condition to a hyperbolic-valued $\Delta_{bc}$-polyharmonic function $F$ of order $\ell$ to be the hyperbolic real part of some $(m,n)$-bc-polyholomorphic function is to satisfy $\partial_{Z^\dagger} F= \partial_{\widetilde{Z}} F=0$ with $\ell=\max(m,n)$.
%\end{thm}

\begin{proof}[Proof of Theorem \ref{thmCharmn}]
	Starting from a given $f= f^+(\alpha) \eu +  f^-(\beta) \es\in \mathcal{A}^{[1]}_{m,n}(\BC)$ with $ f^+\in A_m^\alpha[\overline{\alpha}]$ and $ f^-\in A_n^\beta[\overline{\beta}]$, it is clear 
	that the functions $\psi^+=\Re(f^+) $ and $\psi^-=\Re(f^-)$ are $\Delta_\alpha$- and $\Delta_\beta$-polyharmonic of order $m$ and $n$, respectively.
	Therefore, $ \Re_{hyp}(f) 
	= \Re(f^+)\eu+\Re(f^-)\es =    \psi^+ \eu + \psi^- \es$ and $\Delta_{bc}^{\max(m,n)}( F) = \Delta_{bc}^{\max(m,n)}( \psi^+\eu+ \psi^-\es )=0.$ For the converse, let $F: \BC \longrightarrow \D$ be a given $\Delta_{bc}$-polyharmonic function  of order $\ell$. The assumption  $\partial_{Z^\dagger} F= \partial_{\widetilde{Z}} F=0$ ensures the independence of the idempotent competent $F^+$ (resp.  $F^-$) in $\beta$ and $\overline{\beta}$ (resp. $\alpha$ and $\overline{\alpha}$).  Therefore, one concludes by applying Theorem \ref{polyHH} to the $\Delta_\alpha$- (resp. $\Delta_\beta$-) polyharmonic function $F^+ (resp. F^-): \C\longrightarrow \R$.
\end{proof}

The next result discusses the bc-polyharmonicity of hyperbolic real part of a given bc-polyholomorphic function (whose proof result can be handled in many ways).

\begin{prop} \label{proppolholharm}
	Let $f: \BC\rightarrow \BC$ be a $(m,n,k)$-bc-polyholomorphic function in $\BC$. Then, the following assertions hold trues. 
	\begin{enumerate}
		
		\item[(i)] $f$ is $\Delta_{bc}$-bc-polyharmonic of order $m$.   
		
		\item[(ii)]  $f^{\dagger}$ is $\Delta_{bc}$-bc-polyharmonic of order  $\min(n,k)$ (equivalently $f$ is  $ \Delta_{bc}^{\dagger}$-bc-polyharmonic of order $\min(n,k)$).
		%, where $\Delta_{bc}^\dagger$ is the conjugate Laplacian $\Delta_{bc}^\dagger := \partial_{\widetilde{Z}}\partial_{Z^\dagger}$).
		
		\item[(iii)] The idempotent component function $f^+$ is $\Delta_{\alpha}$-bc-polyharmonic of order $m$, while $f^-$ is $\Delta_{\alpha}$-bc-polyharmonic of order 
		$\min(m,n,k)$.  
		
		\item[(iv)] $\Re_{hyp} (f)$ is $\Delta_{bc}$-bc-polyharmonic of order $m$.
		
		\item[(v)] The function  $\Re_{hyp}(f^\dagger)$ is $\Delta_{bc}$-polyharmonic of order  $\min(n,k)$.	
		
		\item[(vi)] $\Re_c (f) $ is bc-polyharmonic of order $\max(m, \min(n,k))$, with respect to  $\Delta_\alpha$, $\Delta_\beta$ and $\Delta_{bc}$. 
	\end{enumerate} 	
\end{prop}

\begin{proof} 
	Assertion $(i)$ readily follows from $f$ being a $(m,n,k)$-bc-polyholomorphic function, % we have  $\partial_{Z^*}^m f = 0$. 
	$\Delta_{bc}^m f = \partial_{Z}^m\partial_{Z^*}^m f =0$. For $(ii)$ we use the fact 
	$$\Delta_{bc}^\ell (f^\dagger) = (  ( \Delta_{bc}^\dagger )^\ell f ) ^\dagger =   \partial_{Z^\dagger}^\ell \partial_{\widetilde{Z}}^\ell f$$
	to get 
	$\Delta_{bc}^\ell (f^\dagger)=0$ for $\ell \geq \min(n,k)$ and $\Delta_{bc}^{\ell-1} (f^\dagger)\ne0$.
	This proves that $f^\dagger$ is also a bc-polyharmonic function of order $\min(n,k)$. 
	
	The proof of $(iii)$ follows making use of the idempotent decomposition. Indeed,  $f\in \mathcal{A}^{[2]}_{m,n,k}(\BC)$  reads equivalently  $\partial_{\overline{\alpha}}^m f^+ = \partial_{\beta}^n f^+ =\partial_{\overline{\beta}}^k f^+=0$ and  $\partial_{\overline{\alpha}}^m f^- = \partial_{\beta}^n f^- =\partial_{\overline{\beta}}^k f^-=0$. These two systems imply $\Delta_{\alpha}^m f^+ = \Delta_{\alpha}^m f^- =0$. Hence $f^+$ is a $\Delta_{\alpha}$-bc-polyharmonic function of order $m$.
	Moreover, we get  $\Delta_{\alpha}^{\min(n,k)} f^- = \Delta_{\beta}^{\min(n,k)}f^+=0$, which proves that $f^-$ is $\Delta_{\alpha}$-bc-polyharmonic of order 
	$\min(m,n,k)$.  
This can also be handled using Proposition \ref{propChar}, so that one gets  
	\begin{align*}
		\Delta_{\beta}^{r}f^{+} &= \sum_{\ell_{1},\ell_{2},\ell_{3}=0}^{m-1,n-1,k-1}
		\frac{\ell_{2}!}{(\ell_{2}-r)!}\frac{\ell_{3}!}{(\ell_{3}-r)!}
		\overline{\alpha}^{\ell_{1}}
		\overline{\beta}^{\ell_{2}-r} \beta^{\ell_{3}-r} h_{\ell_{1},\ell_{2},\ell_{3}}(\alpha) 
	\end{align*}
	and 
	\begin{align*}
		\Delta_{\alpha}^{r}f^{-} &= \sum_{\ell_{1},\ell_{2},\ell_{3}=0}^{m-1,n-1,k-1}
		\frac{\ell_{2}!}{(\ell_{2}-r)!}\frac{\ell_{3}!}{(\ell_{3}-r)!} \overline{\beta}^{\ell_{1}} \overline{\alpha}^{l_{2-r}}
		\alpha^{\ell_{3}-r} g_{\ell_{1}\ell_{2}\ell_{3}}(\beta).	
	\end{align*}
	With this we can provide another proof of $(i)$ and $(ii)$. 
	
	For the proof of $(iv)$, it should be noted that the involved function $(\alpha,\beta) \longmapsto f^{+}(\alpha e^{+}+\beta e^{-}) $
	is  a polynomial in the variable $\beta$ and $\overline{\beta}$ of exact degrees 
	$k-1$ and $n-1$, respectively. Moreover, it is  $m$-polyharmonic with respect to $\alpha$,  $\alpha \longmapsto f_{\beta}^{+}(\alpha) := f^{+}(\alpha e^{+}+\beta e^{-}) \in \mathcal{A}_{\alpha}^{m}[\overline{\alpha}]$. 
	Accordingly, 
		\begin{align*}
	\Delta_{bc}^{m}(\Re_{hyp} (f))	&=
	\Re (\Delta_{\alpha}^{m} f_{\beta}^{+} )e^{+}+ \Re (\Delta_{\beta}^{m} f_{\alpha}^{-} )e^{+} = 0.
		\end{align*}
	Therefore, $\Re_{hyp} f$ is $\Delta_{bc}$-bc-polyharmonic of order $m$.  This can also be reproved making use of $(i)$ and $(ii)$. Indeed, we have 
	$\Delta_{bc}^\ell (\Re_{hyp} f ) = \Re_{hyp} (\Delta_{bc}^\ell f)=0,$
	for $\ell \geq m$, since $2\Re_{hyp} (f) = f +f^*$ and $ \Delta_{bc} f^* = \left( \Delta_{bc} f\right) ^*$, which proves $(iv)$. 
	
	Assertion $(v)$ follows from $(ii)$ combined with the fact that  $$\Delta_{bc}\left(\Re_{hyp} f^\dagger \right) = \Re_{hyp} ( \Delta_{bc}^\dagger (f) )^\dagger .$$

	Finally, the proof of $(vi)$ lies in the polynomiality of $ \Re(f_{\beta}^+)$ and the polynomiality of $ \Re(f_{\beta}^-)$, which implies
	$$\Delta_{\beta}^{\min (n,k)}(\Re f_{\beta}^{+})=0  \quad \mbox{ and } \quad \Delta_{\alpha}^{(\min (n,k))}(\Re f_{\alpha}^{-})=0,$$ respectively. 
	Subsequently, for $s=\max\left(m,\min(n,k)\right)$ we have 
	\begin{align*}
		\Delta_{\alpha}^{s} ( \Re f_{\beta}^{+}+\Re f_{\alpha}^{-} )
		&=\Delta_{\alpha}^{s}(\Re f_{\beta}^{+})+  \Delta_{\alpha}^{s}(\Re f_{\alpha}^{-})=0
	\end{align*}
	and 
	\begin{align*}
		\Delta_{\beta}^{s} ( \Re f_{\beta}^{+}+\Re f_{\alpha}^{-} )
		=\Delta_{\beta}^{s}(\Re f_{\beta}^{+})+  \Delta_{\beta}^{s}(\Re f_{\alpha}^{-})=0.
	\end{align*}
	This proves that $\Re_{c}(f)$ is polyharmonic of order $s= \max (m,\min (n,k))$ with respect to $\Delta_{\alpha}$, $\Delta_{\beta}$ and $\Delta_{bc}$. 
	An alternative direct proof of $\Delta_{bc}$-bc-polyharmonicity of $\Re_c f$ makes use of the observation that 
	$$4\Re_c(f)= f+f^*+f^\dagger+\widetilde{f}= 2(\Re_{hyp} (f) + (\Re_{hyp} (f))^\dagger),$$ and making appeal to $(iv)$ together with $(ii)$.
\end{proof}

\begin{rem}	
	For $f\in \mathcal{A}^{[2]}_{m,n,k}(U)$, the function $\Re_{c}(f)$ is then harmonic with respect to
	$ \Delta_{\alpha}^{s}+\Delta_{\beta}^{s}$ with $s= \max (m,\min (n,k))$.
	Notice also that
	$$ \Delta_{bc}^{r}(\Re_c (f))  = 	\Delta_{bc}^{r}(\Re_{hyp} (f))  +  \Delta_{bc}^{r}\left((\Re_{hyp} (f))^\dagger\right) . $$ 
	Therefore, if $\Re_{hyp} (f)$ and $(\Re_{hyp} (f))^\dagger$ are bc-polyhamonic, then $\Re_{c}(f)$ is also  bc-polyharmonic. 
\end{rem}

\begin{rem}	
	Let $F: \BC \longrightarrow \D$ such that $F=\Re_{hyp}g$ for given bicomplex-valued function $g$. A sufficient condition for $F$ to be bc-polyharmonic of order $m$ is that $\partial_{Z^*}^{m}g =0$.
\end{rem}	

\begin{prop}\label{propunic}
	For $f\in \mathcal{A}^{[2]}_{m_f,n_f,k_f}(U)$ and $g\in \mathcal{A}^{[2]}_{m_g,n_g,k_g}(U)$ of exact order $(m_f,n_f,k_f)$ and $(m_g,n_g,k_g)$, respectively, such that $\Re_{hyp} f = \Re_{hyp} g$, then $(m_f,n_f,k_f) = (m_g,n_g,k_g)$.		
\end{prop}

\begin{proof} Let $f$ and $g$ be as in Proposition \ref{propChar}. Then from $(i)$ and $(ii)$ we get
	$m_f=m_g$ and $\min(n_f,k_f) = \min(n_g,k_g)$. Accordingly, without loss of generality, we can assume that $\min(n_f,k_f) = n_f$. Thus, we distinguish two cases. 
	The first case of $n_f=k_g$ implies $n_f \leq \min(k_f,k_g)$. Therefore by applying the operator $\partial_{Z^\dagger}^{n_f}$ to both sides of $\Re_{hyp} (f) = \Re_{hyp} (g)$, keeping in mind the facts $\partial_{Z^\dagger}^{n_f} f=0$, $\partial_{Z^\dagger} f = (\partial_{\widetilde{Z}} f)^*$ and $\partial_{Z^\dagger}^{n_f}g = \partial_{Z^\dagger}^{n_g}g =0$, we obtain 
	$$ (\partial_{\widetilde{Z}}^{n_f} f )^* = \partial_{\widetilde{Z}}^{n_f} g  .$$
	Thus, by the action of $\partial_{\widetilde{Z}}^{\ell} $ for $\ell=k_f-n_f$ (resp. $\ell= k_g-n_g$), it readily follows that $0=  \partial_{\widetilde{Z}}^{n_f} g$ (resp. $\partial_{\widetilde{Z}}^{n_f} f=0$). This shows that $k_f \leq k_g$ (resp. $k_f\geq k_g$). Hence $k_f=k_g$.
	
	The second case of $n_f=n_g$ can be handled in a quite similar way.  Indeed, by making use of $\partial_{Z^\dagger}^{k_f} $, we get from $\Re_{hyp} (f) = \Re_{hyp} (g)$ 
	that $0=  \partial_{Z^\dagger}^{k_f} g$ since $ k_f\geq n_f$ and $ \partial_{Z^\dagger}^{k_f} f= 0= \partial_{Z^\dagger}^{k_f} f^* = \partial_{Z^\dagger}^{k_f} g^*$. 
	Hence $k_f \geq n_g$. But, by applying $\partial_{Z^\dagger}^{k_g} $ we obtain $ \partial_{Z^\dagger}^{k_g} f^* = \partial_{Z^\dagger}^{k_g} g$ which gives rise to 
	$ ( \partial_{\widetilde{Z}}^{n_g}f)^*=0$. This shows that $n_g \geq k_f$. This completes the proof of the requested assertion.
\end{proof}

The similar question for the subclass $\mathcal{A}^{[1]}_{m,n}(\BC)$ seems to require further investigation. To this end, we establish the following.

\begin{lem}\label{lempolholharm5}
	Let $F :\BC \longrightarrow \D$  be a $\Delta_{bc}$-polyharmonic function of order $m$. Then, there exist nonnegative integers $r,s$ such that $ m=\max(r,s)$,  and a bicomplex-valued function $G$ in $\BC$ satisfying 
	\begin{enumerate}
		\item[(i)] $G= \Re_{hyp} (F)$ and therefore $\Delta_{bc}$-polyharmonic.
		\item[(ii)] $ G$ belongs to $\ker(\partial_{Z^*}^m)$. 
	\end{enumerate}
\end{lem}

\begin{proof}
	Consider the bicomplex-valued function $G(\alpha\eu+\beta\es) :=  \varphi_\beta(\alpha)\eu + \phi_\alpha(\beta) \es ,$
	where $\varphi_\beta$ and $\phi_\alpha$ are those involved in the proof of Proposition \ref{proppolholharm}. Then, $G$ belongs to $  A_r^\alpha[\overline{\alpha}] \eu + A_s^\beta[\overline{\beta}] \es \subset A_m^\alpha[\overline{\alpha}] \eu + A_m^\beta[\overline{\beta}] \es  $.
	It is clear that $G$ is $\Delta_{bc}$-bc-polyharmonic which proves $(i)$.
	Moreover, we have  
	$$ \Re_{hyp}(G) =   \Re(\varphi_\beta)(\alpha) \eu  + \Re(\phi_\alpha)(\beta) \es = F .$$
	Next, by expanding  $\varphi_\beta \in A_r^\alpha[\overline{\alpha}]$ and $\phi_\alpha \in A_s^\beta[\overline{\beta}]$  
	as 
	$$ \varphi_\beta(\alpha) = \sum_{j=0}^{r-1}  \overline{\alpha}^j \varphi_{\beta,j}(\alpha) \quad \mbox{and} \quad \phi_\alpha(\beta) = \sum_{k=0}^{s-1}  \overline{\beta}^k \phi_{\alpha,k}(\beta)$$
	for some holomorphic functions 
	$\varphi_{\beta,j}$ and $  \phi_{\alpha,k}$, the function 
	$G$ can be rewritten as 
	\begin{align*}
		G_{r,s}(\alpha\eu+\beta\es) &=\frac 12 \left( \sum_{j=0}^{r-1}  \overline{\alpha}^j \varphi_{\beta,j}(\alpha) \eu +  \sum_{k=0}^{s-1}  \overline{\beta}^k \phi_{\alpha,k}(\beta) \es \right) \\
		&=  \sum_{j=0}^{\max(r,s)-1} {Z^*}^j \psi_j(\alpha\eu+\beta\es) \in \ker (\partial_{Z^*}^{\max(r,s)}) 
	\end{align*}
	for given bicomplex-valued functions $\psi_j$. This completes the proof of assertion  $(ii)$.
\end{proof}

With the above data, we can give a succinct proof of the theorem below.

\begin{thm}\label{thmpolholharm5}
	Let $F :\BC \longrightarrow \D$  be a $\Delta_{bc}$-polyharmonic function of order $m$. There exists a unique 
	pair of nonnegative integers $(r,s)$ such that $ m=\max(r,s)$ and if $f\in \mathcal{A}^{[1]}_{r',s'}(\BC)$ whose hyperbolic part is $F$, then $r'=r$ and $s'=s$. 
\end{thm}

\begin{proof}%[Proof of Theorem \ref{thmpolholharm5}]
	Let $F:\BC \longrightarrow \BC\in \mathcal{A}^{[1]}_{r',s'}(\BC) $ and write 
	$ F=F^+(\alpha)\eu +F^-(\beta)\es$ with $F^+\in  \mathcal{A}^{\alpha}_{r'}[\overline{\alpha}]$ and $F^- \in \mathcal{A}^{\beta}_{s'}[\overline{\beta}]$, such that $f=\Re_{hyp} (f) = \Re (F^+ )(\alpha) \eu + F^-(\beta) \es$. 
	According to the proof of Lemma \ref{lempolholharm5}, we conclude that $\Re (F^+ )(\alpha) = \Re \varphi_\beta (\alpha) $ and $\Re (F^- )(\beta) = \Re (\phi_\alpha)(\beta) $
	for every $\alpha,\beta\in \C$.
	Therefore, the polyharmonicity of the idempotent component functions $F^+$, $F^-$, $\varphi_\beta$ and $\phi_\alpha$ shows that $r=r'$ and $s=s'$.   
\end{proof}

We conclude this paper with a proof of our main result (Theorem \ref{Mainthmp}) concerned with the bicomplex analog of Theorem \ref{polyHH} for the elements belonging to $\mathcal{A}^{[2]}_{m;n,k}(\BC)$.  Namely, we provide appropriate conditions to a hyperbolic-valued bc-polyharmonic function to be the hyperbolic part of certain $(m,n,k)$-bc-polyholomorphic function.  
%thanks to Proposition \ref{proppolholharm}.

\begin{proof}[Proof of Theorem \ref{Mainthmp}]
	Let $F: \BC \longrightarrow \D$ be a bc-polyharmonic function of order $m$  and assume that there exists a $(s,n,k)$-bc-polyholomorphic function $f \in \mathcal{A}^{[2]}_{m,n,k}(\BC)$ such that $F=\Re_{hyp} (f)$. Then, by Proposition \ref{propunic}, we have $s=m$. Moreover, a direct computation shows that for $\ell \geq \max(n,k)$ we have 
	$$\partial_{Z^\dagger}^\ell F = \frac 12 \left(  \partial_{Z^\dagger}^\ell f + (\partial_{\widetilde{Z}}^\ell f )^* \right)  = 0$$
	and  
	$$\partial_{\widetilde{Z}}^\ell F  = \left( \partial_{Z^\dagger}^\ell F \right)^* =0.$$ 
	Accordingly, $F$ belongs to $\ker(\partial_{\widetilde{Z}}^{\max(n,k)}) \cap \ker(\partial_{Z^\dagger}^{\max(n,k)})$.  
	
	Conversely, let $F: \BC \longrightarrow \D$ be a bc-polyharmonic function of order $m$ such that $F\in \ker(\partial_{Z^\dagger}^{n})\cap \ker(  \partial_{\overline{Z}}^{k})$. Then, $F$ is of the form 
	$$ F(Z) = \sum_{\ell_1=0}^{n-1}\sum_{\ell_2=0}^{k-1} G_{\ell_1,\ell_2}(Z,Z^*) (Z^\dagger)^{\ell_1} (\widetilde{Z})^{\ell_2}$$
	for some bicomplex-valued-functions $G_{\ell_1,\ell_2}$, which becomes bc-polyharmonic of order $m$ by assuming that $\Delta_{bc}^m F=0$. 
	Now, using the idempotent decomposition we get the system $\partial_{\beta}^n R^+= \partial_{\overline{\beta}}^k F^+= 0$ and $\partial_{\alpha}^n F^- =\partial_{\overline{\alpha}}^k F^-=0$. Thus, we can rewrite $F(Z) = F^+(Z) \eu + F^-(Z)\es$ with $Z= \alpha\eu + \beta\es$ as
	$$ F^+(Z) = \sum_{\ell_1=0}^{n-1}\sum_{\ell_2=0}^{k-1} a_{\ell_1,\ell_2}(\alpha,\overline{\alpha} ) \beta^{\ell_1} (\overline{\beta})^{\ell_2}$$
	and 
	$$ F^-(Z) = \sum_{\ell_1=0}^{n-1}\sum_{\ell_2=0}^{k-1} b_{\ell_1,\ell_2}(\beta,\overline{\beta}) \alpha^{\ell_1} (\overline{\alpha})^{\ell_2},$$
	so that 
	$$ F(Z) = \sum_{\ell_1=0}^{n-1}\sum_{\ell_2=0}^{k-1} \left( a_{\ell_1,\ell_2}(\alpha,\overline{\alpha} ) \eu + b_{\ell_1,\ell_2}(\beta,\overline{\beta}) \es \right)  (Z^\dagger)^{\ell_1} (\widetilde{Z})^{\ell_2},$$
	where $a_{\ell_1,\ell_2}(\alpha,\overline{\alpha} )$ and $b_{\ell_1,\ell_2}(\beta,\overline{\beta}) $ are complex-valued functions on $\C$. Now, taking into account that $F$ is an $m$-bc-polyharmonic function, we conclude  
	$$ \sum_{\ell_1=0}^{n-1}\sum_{\ell_2=0}^{k-1} \Delta_{bc}^m \left( a_{\ell_1,\ell_2}(\alpha,\overline{\alpha} ) \eu + b_{\ell_1,\ell_2}(\beta,\overline{\beta}) \es \right)  (Z^\dagger)^{\ell_1} (\widetilde{Z})^{\ell_2} = 0.$$
	Therefore $a_{\ell_1,\ell_2}(\alpha,\overline{\alpha} )$ and $b_{\ell_1,\ell_2}(\beta,\overline{\beta}) $ are  polyharmonic of order $r$ and $s$ with respect to $\Delta_\alpha$ and $\Delta_\beta$, respectively, such that  $m=\max(r,s)$. By identification, it follows 
	$$  G_{\ell_1,\ell_2}(Z,Z^*) = a_{\ell_1,\ell_2}(\alpha,\overline{\alpha} ) \eu + b_{\ell_1,\ell_2}(\beta,\overline{\beta}) \es .$$
	In the other hand, by making use of Theorem \ref{polyHH}, there are $m$-polynanalytic functions $\varphi_{\ell_1,\ell_2} \in  \mathcal{A}^{\alpha}_{r}[\overline{\alpha}]$ and $\psi_{\ell_1,\ell_2}\in \mathcal{A}^{\beta}_{s}[\overline{\beta}]$ such that   
	$a_{\ell_1,\ell_2} = \Re(\varphi_{\ell_1,\ell_2})$ and $  b_{\ell_1,\ell_2}= \Re(\psi_{\ell_1,\ell_2})$. Hence 
	$$ F(Z) = \sum_{\ell_1=0}^{n-1}\sum_{\ell_2=0}^{k-1} \Re_{hyp} f_{\ell_1,\ell_2} (Z,Z^*)   (Z^\dagger)^{\ell_1} (\widetilde{Z})^{\ell_2}.$$
	The involved functions $f_{\ell_1,\ell_2}$ are given by 
	$$ 
	f_{\ell_1,\ell_2} = a_{\ell_1,\ell_2}(\alpha,\overline{\alpha} ) \eu + b_{\ell_1,\ell_2}(\beta,\overline{\beta}) \es  $$ and belong  to $\mathcal{A}^{[1]}_{r,s}(\BC)$. 
\end{proof}

\begin{rem}
	For $n=k=1$ we recover the result of Theorem \ref{thmCharmn}.
\end{rem}

%\subsection*{Acknowledgment}

\noindent {\bf Data availability statement:}	All data generated or analyzed during this study are included in this article.

\section*{Conflict of interest}

The authors declare that they have no conflict of interest.

% ------------------------------------------------------------------------
\end{document}